\documentclass[12pt,twoside,reqno]{amsart}
\usepackage{amsfonts}
\usepackage{amssymb}
\usepackage{amsmath}
\usepackage{bbm}
\usepackage[notcite,notref]{%showkeys
}
\usepackage{graphicx}
\usepackage{pgf,tikz}
\usepackage{float}
\usepackage{subcaption}

\numberwithin{equation}{section}

\newtheorem{teo}{Theorem}[section]
\newtheorem{prop}[teo]{Proposition}

\theoremstyle{definition}

\newtheorem{rem}[teo]{Remark}

\makeatletter
\@namedef{subjclassname@2010}{%
  \textup{2010} Mathematics Subject Classification}
\makeatother

\def\a{\alpha}
\def\b{\beta}

\def\R{\mathbb{R}}

\def\N{\mathbb{N}}

\def\e{\varepsilon}

\def\Re{\mathcal{R}}
\def\cX{\mathcal{X}}

\def\T{\mathcal{T}}
\def\S{\mathcal{S}}
\def\O{\mathcal{O}}

\def\oU{\overline{U}}
\def\ov{\overline{v}}
\def\ow{\overline{w}}

\def\tU{\widetilde{U}}
\def\tv{\widetilde{v}}
\def\tw{\widetilde{w}}

\def\eu{e_{\oU}}
\def\ev{e_{\ov}}
\def\ew{e_{\ow}}

\begin{document}

\title[A priori feedback estimates]{A priori feedback estimates for multiscale reaction-diffusion systems}

\author{Martin Lind}
\address{Department of Mathematics and Computer Science\\Karlstad University\\
651 88 Karlstad\\ Sweden }
\email{martin.lind@kau.se}

\author{Adrian Muntean}
\address{Department of Mathematics and Computer Science\\Karlstad University\\
651 88 Karlstad\\ Sweden }
\email{adrian.muntean@kau.se}



\begin{abstract}
We study  the approximation of a multiscale reaction-diffusion system posed on both macroscopic and microscopic space scales. The coupling between the scales is done via micro-macro flux conditions. Our target system has a typical structure for reaction-diffusion-flow problems in  media with distributed microstructures (also called, double porosity materials).  Besides ensuring basic estimates for the convergence of  two-scale semi-discrete Galerkin approximations, we provide a set of {\em a priori} feedback estimates  and a local feedback error estimator that help  in designing a distributed-high-errors strategy to allow for a computationally efficient zooming in and  out from microscopic structures.  The error control on the feedback estimates relies on  two-scale-energy, regularity, and interpolation  estimates as well as  on a fine bookeeping of the sources responsible with the propagation of the (multiscale) approximation errors. The working technique based on {\em a priori } feedback estimates is in principle applicable to a large class of systems of PDEs with dual structure admitting strong solutions. 
\end{abstract}
\maketitle

\section{Introduction}

Reaction-diffusion systems posed on multiple spatial scales became recently a powerful modelling and simulation tool \cite{Murad,Redeker,Jager2}. Conceptually, multiscale models are very much linked to physical scenarios where averaging procedures (like multiple scales asymptotic expansions, periodic/stochastic homogenization, REV-based methods, renormalization) fail to bring information in a consistent  trustful  way up to observable macroscopic scales. Often, either due to special geometric sub-structures (cf. \cite{Barenblatt}, e.g.), to separated fast-slow characteristic times (cf. \cite{Sanchez,E,Jager1}, e.g.) or to a suitable combination of both such effects, balance laws of extensive physical quantities must be posed on geometries with separated tensor-product space-scale structures (cf. \cite{Showalter,NeussRadu,MunteanNeussRadu}).  

By zooming in and out at the position of a continuum collection of material points, such models allow for a significant enrichment of a  usually rough/coarse macroscopic information from a detailed microscopic picture. Such new multiscale modelling possibilities also open a set of fundamental questions that must be addressed mathematically so that these multiscale models not only get a well-posedness theory, but also are easily accessible through computations. This is the place where our paper contributes. 

We present a two-scale  Galerkin approach for a particular class of multiscale reaction-diffusion systems with linear
coupling between the microscopic and macroscopic variables. Exploiting the special structure of the model,
the functions spaces used for the approximation of the solution are chosen as tensor
products of spaces on the macroscopic domain and on the standard cell associated to the
microstructure. Uniform estimates for the finite dimensional approximations allow us to ensure the convergence of the Galerkin approximates. However, since the zooming in/out must be in principle done for a large number of spatial points of the macroscopic domain to ensure a good quality of the approximation, and moreover, the physics at the microscopic level is quite complex, we are wondering whether we can build mesh refinement strategies in an {\em a priori} fashion to reduce drastically the computational effort up to a minimum level, absolutely  needed to trust the simulation output. 

In the context of finite element methods, so called \emph{adaptive mesh refinement strategies} have been used for a long time. The idea is very natural: starting from a coarse mesh, one computes the approximate solution and a quantity measuring (in some sense) the local error. One then subdivides those elements of the mesh with large local errors. 
Practical experience suggest that iterations of such adaptive refinements generally converge. This was first shown rigorously in the one-dimensional case by Babu\v{s}ka and Vogelius \cite{BabuskaVogelius} using {\em a priori} feedback estimates.
Inspired by \cite{BabuskaVogelius}, there has been a lot of work on {\em a posteriori} analysis of adaptive finite element methods in higher dimensions (see the survey \cite{Nochetto1} and the references therein). 

The aim of this paper is to develop a feedback scheme in the vein of \cite{BabuskaVogelius} for systems of reaction-diffusion equations posed on multiple scales. As discussed above, such an approach is motivated by the desire to avoid 'zooming into' the micro-structure too often.

The main problem we have to overcome is the fact that the errors of approximation on the two scales are coupled, and it is not clear from the outset how this effects the total error. Our main insight is the fact that the errors on the microscopic scale are in some sense controlled by the macroscopic error. This allows us to develop a feedback refinement scheme based on \emph{a priori} errors on the macroscopic scale that ensures convergence of our approximate solutions. We note however that we do not obtain any specific rate of convergence for the approximates. This is to be expected due to the fact that we consider completely general meshes without any special a priori imposed structure. 

Finally, we want to mention the connection to nonlinear approximation matters. Indeed, Galerkin approximation on adaptively generated meshes is an instance of a nonlinear approximation scheme. It is intuitively clear that an adaptive mesh refinement scheme yields a sequence of meshes where the local errors are, roughly speaking, increasingly 'equidistributed'. In this direction, see the discussion in \cite{Nochetto1} and the papers \cite{Binev1,Binev2}.

%More precisely, the \emph{total error of approximation} $\mathcal{E}$ is defined as the sum of the microscopic errors and the error on the macroscopic scale (i.e. the differences between weak solutions and Galerkin solutions measured in suitable norms).
%We prove (see Theorem \ref{micromacroError} below) that the error of approximation for the microscopic solutions is essentially dominated by the the error of approximation of the macroscopic solution $U$. Hence the total error of approximation $\mathcal{E}$ is fully controlled by the macroscopic error and this allows us to construct a refinement scheme at the macroscopic level in the spirit of \cite{BabuskaVogelius}.
%A consequence of the estimator is the following vaguely formulated convergence theorem. 
%\begin{teo}
%\label{mainTeo} 
%If $\mathcal{E}_N$ is the total error of approximation after $N$ iterations of the refinement scheme, then
%\begin{align}
%\mathcal{E}_N=\O(h_Y^2)+\e_N
%\end{align}
%where $\lim_{N\rightarrow\infty}\e_N=0$ and $\O(h^2_Y)$ is the a priori rate of approximation of the microscopic solutions.
%\end{teo}

%In the final section below, we outline some plans for future work.

\subsection{Notations}
For the convenience of the reader, we introduce some standard notation that we shall use.

Below $D_1\subset\R^m$ and $D_2\subset\R^n$ will be arbitrary connected domains. 
For $(x,y)\in D_1\times D_2$ we denote by $\nabla_x,\nabla_y$ the gradients and $\Delta_x,\Delta_y$ the Laplace operators with respect to the indicated variables. Generally we omit the index of the first variable: we shall write $\nabla=\nabla_x$ and $\Delta=\Delta_x$.
Denote by $L^2(D_1)$ the space of measurable functions on $D_1$ such that
$$
\|f\|^2_{L^2(D_1)}:=\int_{D_1}|f|^2<\infty.
$$
The Sobolev space $H^k(D_1)$ consists of all functions with all $k$-th order weak partial derivatives in $L^2(D_1)$. E.g., $f\in H^1(D_1)$ if and only if
$$
\|f\|^2_{H^1(D_1)}:=\|f\|^2_{L^2(D_1)}+\|\nabla_xf\|^2_{L^2(D_1)}<\infty.
$$ 
We shall frequently use function spaces with mixed norms (also called \emph{Bochner spaces}) of different types. To be as general as possible, let $(B,\|\cdot\|_B)$ and $(X,\|\cdot\|_X)$ denote Banach spaces of functions defined on the domains $D_1$, $D_2$ respectively. We say that a function $f(x,y)$ defined on the set $D_1\times D_2$ belongs to the space $X(D_1,B)$ if
$$
\|\;\|f(x,\cdot)\|_B\;\|_X<\infty.
$$
For instance, if $\Omega\subset\R^d$ is a domain, the space $L^2([0,T],H^1(\Omega))$ consists of all functions $f(t,x)$ such that
$$
\int_0^T\|f(t,\cdot)\|^2_{L^2(\Omega)}+\|\nabla_xf(t,\cdot)\|^2_{L^2(\Omega)}dt<\infty.
$$

\section{Setting of the problem}

\subsection{Physical meaning and strong formulation} 

Let $\Omega\subset\R^m$ be a connected domain that has at each $x\in\Omega$ a standard microscopic pore $\mathcal{Y}$. Thus, $\Omega\times\mathcal{Y}$ models a porous material with the macroscopic domain $\Omega$ and microstructure represented by $\mathcal{Y}$, see Figure 1 below)
\begin{figure}[H]
\centering
\includegraphics[scale=0.8]{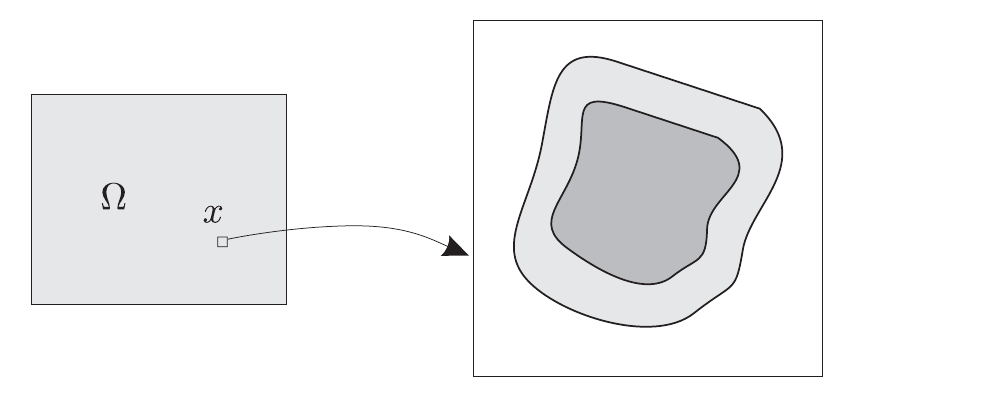}
\caption{The macroscopic domain $\Omega$ and microscopic cell $\mathcal{Y}$ at $x\in\Omega$}
\end{figure} 
We assume that $\mathcal{Y}$ is partially filled with water and partially with gas. Denote by $Y\subset \mathcal{Y}$ be the wet region of the microstructure (light gray in Figure 1) and $Y^g$ the gas-filled part (dark gray in Figure 1). Set $\Gamma=\partial Y$ and denote $\Gamma^R$ the gas-liquid interface (boundary between light and dark gray in Figure 1).
We consider a chemichal species $A_1$ penetrating $\Omega$ through the air-filled part of the pore and dissolves into the water along $\Gamma^R$. In water, $A_1$ transforms to $A_2$ and reacts with the species $A_3$, producing water and other products (typically salts).

We denote by $U$ the macroscopic mass concentration of $A_1$ and by $v,w$ the microscopic mass concentrations of $A_2,A_3$ respectively. This particular micro-macro structure of the model equations has been rigorously derived by mean of periodic homogenization arguments in \cite{Gakuto}. See also \cite{Sebam_CRAS} for a related setting involving freely evolving reaction interfaces inside the periodic microstructure.
%\subsection{Strong formulation}

Denote by $S=(0,T)$ for a given $T>0$. The concentrations $(U,v,w)$ satisfy the following system of equations
\begin{equation}
\label{system}
\left\{\begin{array}{rcl}
\partial_tU-D_U\Delta U&=&\gamma\int_{\Gamma^R}-D_v\nabla_y v(t,x,y)\cdot n_yd\sigma_y\\
\partial_tv-D_v\Delta_yv&=&-\eta(v,w)\\
\partial_tw-D_w\Delta_yw&=&-\eta(v,w).
\end{array} \right.
\end{equation}
Above, $n_y$ is the unit normal of $\Gamma^R$ and $d\sigma_y$ is the surface measure on $\Gamma^R$.
We impose the following boundary conditions
\begin{equation}
\label{systemBoundary}
\left\{\begin{array}{rcl}
U(t,x)&=&U_D(t,x)\quad{\rm at}\quad S\times \partial\Omega\\
-D_v\nabla_y v\cdot n_y&=&\a(v-U)\quad{\rm at}\quad S\times\Omega\times \Gamma^R\\
-D_v\nabla_y v\cdot n_y&=&0\quad{\rm at}\quad S\times\Omega\times\Gamma\setminus\Gamma^R\\
-D_w\nabla_y w\cdot n_y&=&0\quad{\rm at}\quad S\times\Omega\times\Gamma
\end{array} \right.
\end{equation}
and initial conditions
\begin{align}
\label{systemInitial}
U(0,x)=U_I(x),\quad v(0,x,y)=v_I(x,y),\quad w(0,x,y)=w_I(x,y)
\end{align}
for $x\in\Omega$ and $(x,y)\in\Omega\times Y$.

It is worth noting that the coupling between the macroscopic scale and the microscopic scale takes place at two prominent places: On one hand, the coupling is present  in the source/sink term in the mass balance equation governing the evolution of $U$, while on the other hand it appears explicitly in the micro-macro flux condition. %(ADD the BC on $\Gamma^R$)%.
The parameter $\alpha>0$ tunes the transport of mass across air-water interfaces, while the parameter $\gamma>0$ ensure the conservation of mass when the information is transmitted between the micro and macro scales and {\em vice versa}. 

\subsection{List of assumptions}
In this subsection we collect all assumptions on the data given above. 
\begin{enumerate}
\item\label{A1} The domain $\Omega$ is convex and $\partial\Omega$ is Lipschitz;
\item\label{A2} $\Gamma,\Gamma^R$ are Lipschitz and $d\sigma_y(\Gamma^R)>0$ (where $d\sigma_y$ denotes surface measure);
\item\label{A3} $D_U,D_v,D_w$ are positive constants;
\item \label{A4} $\eta(x,y)$ is globally Lipschitz continuous with respect to both variables;
\item\label{A5} the boundary value $U_D$ is the trace %$T|_{\partial\Omega}U_D^*$
 on $\partial\Omega$ of a function $U^*_D\in L^2(S, H^1(\Omega))$;
\item\label{A6} the initial values $U_I,v_I,w_I$ satisfy
$$
(U_I,v_I,w_I)\in H^1(\Omega)\times [L^2(\Omega,H^1(Y))]^2;
$$
\item\label{A7} the initial values and their \emph{Galerkin projections}
satisfy estimates of the type
$$
\|U_I-\widetilde{U_I}\|^2_{L^2(\Omega)}=\mathcal{O}(h^2_\Omega),
$$
and
$$
\|\chi_I-\widetilde{\chi_I}\|^2_{L^2(\Omega, L^2(Y) )}=\mathcal{O}(h^2_Y)\quad{\rm for}\quad\chi\in\{v,w\},
$$
(see Section 3.2 below for an explaination of the previous notation).
\end{enumerate}
We will make some remarks on the assumptions. To be able to lift the spatial regularity in the macroscopic domain, we assume that $\Omega$ is convex and $\partial\Omega$ is Lipschitz (see e.g. \cite{Grisvard}). Likewise, to lift regularity in the micro-domain, we either need $Y$ to be convex with Lipschitz $\partial Y$, or $Y$ can be taken arbitrary but in that case $\partial Y$ must be smoother, e.g. $\partial Y\in C^{2}$. Note that if later on in future approaches $\partial Y$ will be freely evolving in time, then  the convexity assumption on $Y$ is not anymore realistic.

The assumption (\ref{A7}) only states that the initial values may be well-approximated, which we can always assume by taking them smooth enough.

\subsection{Weak formulation}

Our concept of \emph{weak solution} to (\ref{system})-(\ref{systemInitial}) is the following. A triplet $(U,v,w)$ such that $U-U_D^*\in L^2(S,H^1_0(\Omega))$, $\partial_tU\in L^2(S\times\Omega)$, $(v,w)\in L^2(S,L^2(\Omega,H^1(Y)))^2$, $(\partial_tv,\partial_tw)\in L^2(S\times\Omega\times Y)^2$ is called a weak solution of (\ref{system}) if for a.e. $t\in S$ the equations
\begin{align}
\label{weakU}
\int_\Omega\partial_tU\varphi+D_u\int_\Omega\nabla U\nabla\varphi =\gamma\a\int_{\Omega\times\Gamma^R}(v-U)\varphi d\sigma_ydx
\end{align}
\begin{align}
\nonumber
\int_{\Omega\times Y}\partial_tv\psi+D_v\int_{\Omega\times Y}\nabla_yv\nabla_y\psi
+\a\int_{\Omega\times\Gamma^R}(v-U)\psi d\sigma_y dx\\
\label{weakv}
=-\int_{\Omega\times Y}\eta(v,w)\psi
\end{align}
\begin{align}
\label{weakw}
\int_{\Omega\times Y}\partial_tw\phi+D_w\int_{\Omega\times Y}\nabla_y w\nabla_y\phi =-\int_{\Omega\times Y}\eta(v,w)\phi 
\end{align}
hold for all test functions $(\varphi,\psi,\phi)\in H_0^1(\Omega)\times L^2(\Omega,H^1(Y))^2$, and
\begin{align}
\nonumber
U(0)=U_I\,\,{\rm in}\,\,\Omega,\quad (v(0),w(0))=(v_I,w_I)\,\,{\rm in}\,\,\Omega\times Y.
\end{align}
Existence and uniqueness of the weak solution to the system (\ref{weakU})-(\ref{weakw}) was proved in \cite{MunteanNeussRadu}. We state this result here.
\begin{teo}
\label{Existence} There exists a unique weak solution
\begin{align}
\nonumber
(U,v,w)\in  L^2(S,H^1(\Omega))\times\left[L^2(S,L^2(\Omega,H^1(Y)))\right]^2
\end{align}
to (\ref{system}).
\end{teo}

%As expected, imposing more smoothness on the initial values, we obtain higher regularity of the weak solution.
We also have the following regularity lift.
\begin{teo}
\label{highRegularity}
Assume that the initial values of (\ref{system}) satisfy (\ref{A6}).
Then the weak solutions $(U,v,w)$ of (\ref{system}) satisfy
\begin{align}
\label{regLift}
(U,v,w)\in L^2(S,H^2(\Omega))\times [L^2(S,L^2(\Omega,H^2(Y)))]^2.
\end{align}
\end{teo}
We give the proof of the above theorem in Appendix A.
\section{Preliminaries}

\subsection{Auxiliary results}
For the reader's convenience, we collect in this subsection a few standard inequalities that we shall need.

The following very elementary inequality will be very useful. Let $\e>0$ be an arbitrary parameter, then
\begin{eqnarray}
\label{epsilonAMGM}
2|ab|\le \e a^2+\frac{b^2}{\e}.
\end{eqnarray}
Further, we have the following \emph{interpolation-trace inequality}: assume that $Y$ is a Lipschitz domain and that $f\in L^2(\Omega, H^1(Y))$, then for any given parameter $\rho>0$ we have
\begin{align}
\label{InterpolationTrace}
\int_{\Omega\times\partial Y}f^2d\sigma_ydx\le\rho\int_{\Omega\times Y}|\nabla_yf|^2+c_\rho\int_{\Omega\times Y}|f|^2.
\end{align}

\subsection{Galerkin approximation}

We shall discuss briefly the finite-dimensional approximation to (\ref{system}). The discussion will intentionally be rather terse. We will use piecewise polynomial functions.

%, see Remark \ref{defense} below.

Let $D\subset\R^2$ be a polygonal domain (the assumption $D\subset\R^2$ is not necessary, it only makes the terminology simpler). A \emph{partition} $\T=\{\Delta\}$ of $D$ is a finite collection of convex polygonal sets with disjoint interiors such that
\begin{align}
\nonumber
\Omega=\bigcup_{\Delta\in\T}\Delta.
\end{align}
We denote by $S_0(\T)$ the set of all functions $s$ such that $s\in H^2(D)$ and the restriction of $s$ to any $\Delta$ is a polynomial of degree at most $k$ for some fixed pre-specified $k\in\N$. Clearly $\S_0(\T)$ is a finite-dimensional space. For more details on the finite element method, see e.g. \cite{Ciarlet}.

A partition $\T'$ is said to be a \emph{refinement} of $\T$ if for each $\Delta\in\T$ there is unique subset $\Lambda\subset\T'$ such that
\begin{align}
\nonumber
\Delta=\bigcup_{\Delta'\in\Lambda}\Delta'
\end{align}
If $\T'$ is a refinement of $\T$, then
\begin{align}
\nonumber
\S_0(\T)\subset\S_0(\T').
\end{align}

Let $\T_\Omega=\{\Delta\}$ and $\T_Y=\{\Delta'\}$ be partitions of $\Omega$ and $Y$.  Define
\begin{align}
\label{coarseness}
h_\Omega=\max_{\Delta\in\T_\Omega}{\rm diam}(\Delta)\quad{\rm and}\quad h_Y=\max_{\Delta'\in\T_Y}{\rm diam}(\Delta'),
\end{align} 
where ${\rm diam}(E)=\sup_{x,y\in E}|x-y|$. Consider $\S_0(\T_\Omega)$ and $\S_0(\T_Y)$ and let $\{\varphi_i\}\subset H^2(\Omega)\cap H_0^1(\Omega)$ and $\{\psi_j\}\subset H^2(Y)$ be bases for these spaces respectivelty.
Define $\S(\T_\Omega)$ and $\S(\T_Y)$ the set of all functions of the forms
\begin{align}
\nonumber
\sum a_i(t)\varphi_i(x)
\end{align} 
and
\begin{align}
\nonumber
\sum b_{ij}(t)\varphi_i(x)\psi_j(y).
\end{align}
The \emph{Galerkin projections} of $U,v,w$ are the functions
\begin{align}
\label{ProjU}
\tU(t,x)=\sum a_i(t)\varphi_i(x),\\
\label{ProjV}
\tv(t,x,y)=\sum b_{ij}(t)\varphi_i(x)\psi_j(y)\\
\label{ProjW}
\tw(t,x,y)=\sum c_{ij}(t)\varphi_i(x)\psi_j(y)
\end{align}
that satisfy
\begin{align}
\label{GalU}
\int_\Omega\partial_t\tU\varphi+D_u\int_\Omega\nabla \tU\nabla\varphi =\gamma\a\int_{\Omega\times\Gamma^R}(\tv-\tU)\varphi d\sigma_ydx
\end{align}
\begin{align}
\nonumber
\int_{\Omega\times Y}\partial_t\tv\psi+D_v\int_{\Omega\times Y}\nabla_y\tv\nabla_y\psi
+\a\int_{\Omega\times\Gamma^R}(\tv-\tU)\psi d\sigma_y dx\\
\label{GalV}
=-\int_{\Omega\times Y}\eta(\tv,\tw)\psi
\end{align}
\begin{align}
\label{GalW}
\int_{\Omega\times Y}\partial_t\tw\phi+D_w\int_{\Omega\times Y}\nabla_y\tw\nabla_y\phi =-\int_{\Omega\times Y}\eta(\tv,\tw)\phi 
\end{align}
for all $(\varphi,\psi,\phi)\in\S(T_\Omega)\times[\S(\T_\Omega)\times\S(\T_Y)]^2$. The system (\ref{GalU})-(\ref{GalW}) has a unique solution $(\widetilde{a},\widetilde{b},\widetilde{c})\in C^1(0,T)^{N_1}\times \left[C^1(0,T)^{N_1N_2}\right]^2$, where $N_1={\rm dim}\S(\T_\Omega)$ and $N_2={\rm dim}\S(\T_Y)$; see \cite{MunteanNeussRadu} for an argument that carries over to our slightly different setting.

%\begin{rem}
%\label{defense}
%As an example of decompositions $(\T_\Omega,\T_Y)$, one could let $(\T_\Omega,\T_Y)$ to be so-called conforming triangulations of $\Omega$ and $Y$ respectively, and take $\{\varphi_i\}$ and $\{\psi_j\}$ to be bases for the spaces of splines (subordinate to $\T_\Omega,\T_Y$) that have smoothness $r=2$ (i.e. belong to $H^2$) and sufficiently high degree, see e.g. \cite{AlfredPiperSchumaker}.

%Our aim is not to discuss in detail the subtle question of constructing stable bases of smooth splines on triangulations. 
%\end{rem}

On $H^1(\Omega)$ and $L^2(\Omega,H^1(Y))$ we consider the inner products
\begin{align}
\label{inProd1}
\langle \varphi,\psi\rangle_{H^1(\Omega)}=\int_\Omega\varphi\psi+\nabla \varphi\nabla \psi,
\end{align}
and
\begin{align}
\label{inProd2}
\langle \varphi,\psi\rangle_{L^2(\Omega,H^1(Y))}=\int_{\Omega\times Y}\varphi\psi+\nabla_y \varphi\nabla_y \psi.
\end{align}
Given a subspace $V\subset H^1(\Omega)$, denote by $V^\perp$ the orthogonal complement of $V$ with respect to $\langle\cdot,\cdot\rangle_{H^1(\Omega)}$, and similarly for any subspace $W\subset L^2(\Omega,H^1(Y))$.

%We assume that the following holds: given any $f^1\in\S(\T_\Omega)$, $g^1\in L^2(\Omega)\times\S(\T_Y)$, $f^2\in\S(\T_\Omega)^\perp$ and $g^2\in L^2(\Omega)\times\S(\T_Y)^\perp$, there holds for $k\in\{0,1,2\}$

%\begin{align}
%\label{OmegaDec}
%\|f^1\|_{H^k(\Omega)}+\|f^2\|_{H^k(\Omega)}\le C\|f^1+f^2\|_{H^k(\Omega)}
%\end{align}
%and
%\begin{align}
%\label{YDec}
%\|g^1\|_{L^2(\Omega,H^k(Y))}+\|g^2\|_{L^2(\Omega,H^k(Y))}\le %C\|g^1+g^2\|_{L^2(\Omega,H^k(Y))}.
%\end{align}
%It follows from (\ref{OmegaDec}) and (\ref{YDec}) that for $k=0,1,2$ and a.e. $t\in S$ there hold
%\begin{align}
%\label{GalStabil1}
%\|\tU\|_{H^k(\Omega)}\le C\|U\|_{H^k(\Omega)},
%\end{align}
%and
%\begin{align}
%\label{GalStabil2}
%\|\tv\|_{L^2(\Omega,H^k(Y)}\le C\|v\|_{L^2(\Omega,H^k(Y)},\quad\|\tw\|_{L^2(\Omega,H^k(Y)}\le C\|w\|_{L^2(\Omega,H^k(Y)}.
%\end{align}
%\begin{defn}
%\label{admissible}
%A pair of decompositions $(\T_\Omega,\T_Y)$ is said to be \emph{admissible} if there exist finite sets $\{\varphi_i\}\subset H^2(\Omega)$ and $\{\psi_j\}\subset H^2(Y)$  of linearly independent functions, and if (\ref{OmegaDec}), (\ref{YDec}) hold for the corresponding spaces $\S(\T_\Omega), \S(\T_Y)$.
%\end{defn}

\section{Basic a priori error estimates}

In this section we obtain our main error estimates. Throughout, we assume that the initial values $U_I\in H^1(\Omega)$ and $v_I,w_I\in L^2(\Omega,H^1(Y))$ in order to ensure that (\ref{regLift}) holds.

\subsection{Convergence rates}

Let $\T_\Omega,\T_Y$ be arbitrary partitions. Let $\tU\in\S(\T_\Omega)$ and $\tv,\tw\in\S(\T_\Omega)\times\S(\T_Y)$ be the Galerkin projections (\ref{ProjU})-(\ref{ProjW}) and define
\begin{align}
\nonumber
e_U:=U-\tU,\quad e_v:=v-\tv,\quad e_w:=w-\tw.
\end{align}
Note that
\begin{eqnarray}
\nonumber
H^1(\Omega)&=&\S(\T_\Omega)\oplus\S(\T_\Omega)^\perp\\
\nonumber
\quad L^2(\Omega,H^1(Y))&=&[L^2(\Omega)\times\S(\T_Y)]\oplus [L^2(\Omega)\times\S(\T_Y)^\perp],
\end{eqnarray}
where $\S(\T_\Omega)^\perp$ and $\S(\T_Y)^\perp$ denote the orthogonal complements in $H^1(\Omega)$ and $H^1(Y)$, respectively (the inner products are defined by (\ref{inProd1})) and (\ref{inProd2})). We may write
\begin{align}
\nonumber
e_U=e_U^1+e_U^2,\quad e_v=e_v^1+e_v^2,\quad e_w=e_w^1+e_w^2,
\end{align}
where $e_U^1\in\S(\T_\Omega)$ and $e_U^2\in\S(\T_\Omega)^\perp$ and
\begin{align}
\nonumber
e_v^1,e_w^1\in\S(\T_\Omega)\times\S(\T_Y),\quad e_v^2,e_w^2\in L^2(\Omega)\times\S(\T_Y)^\perp
\end{align}
%By (\ref{OmegaDec}),
%\begin{align}
%\label{UerrorLess}
%\|e_U^j\|_{H^k(\Omega)}\le C\|e_U\|_{H^k(\Omega)}\quad j\in\{1,2\},\, k\in\{0,1,2\},
%\end{align}
%and by (\ref{YDec}) we have for $\chi\in\{v,w\}$
%\begin{align}
%\label{VWerrorLess}
%\|e_\chi^j\|_{L^2(\Omega,H^k(Y))}\le C\|e_\chi\|_{L^2(\Omega,H^k(Y))}
%\quad j\in\{1,2\},\, k\in\{0,1,2\}.
%\end{align}

\begin{prop}
\label{RieszProjections}
Let $\T_\Omega,\T_Y$ be arbitrary partitions and assume that the weak solution $(U,v,w)$ satisfies
\begin{align}
\nonumber
(U,v,w)\in L^2(S,H^2(\Omega))\times L^2(S,L^2(\Omega,H^2(Y))^2.
\end{align}
Denoting
\begin{align}
\label{projections}
\Re_U=\tU+e_U^1,\quad \Re_v=\tv+e_v^1,\quad\Re_w=\tw+e_w^1,
\end{align}
we have for $k\in\{0,1\}$
\begin{align}
\label{aprioriU1}
\|U-\Re_U\|_{L^2(S,H^{k}(\Omega))}\le Ch^{2-k}_\Omega\|U\|_{L^2(S,H^2(\Omega))}
\end{align}
and $\chi\in\{v,w\}$
\begin{align}
\label{apriorivw1}
\|\chi-\Re_\chi\|_{L^2(S,L^2(\Omega, H^k(Y))}\le Ch^{2-k}_Y\|\chi\|_{L^2(S,L^2(\Omega,H^2(Y)))},
\end{align}
where $C$ is an absolute constant.
\end{prop}
\begin{proof}
We prove (\ref{aprioriU1}) first. For a.e. $t\in S$ we have
\begin{align}
\nonumber
\langle U-\Re_U,s\rangle_{H^1(\Omega)}=0\quad{\rm for\,\, all}\quad s\in\S(\T_\Omega).
\end{align}
Then it follows from standard estimates on elliptic projections (see e.g. \cite[p.65]{LarssonThomee}) that for a.e. $t\in S$ there holds
\begin{align}
\nonumber
\|(U-\Re_U)(t)\|_{H^k(\Omega)}\le Ch^{2-k}\|U(t)\|_{H^2(\Omega)}\quad (k=0,1),
\end{align}
where $C$ is independent of $t$. Integrating the previous estimate gives (\ref{aprioriU1}).
%hence for each $t\in S$ $\Re_U$ is the best approximation of $U$ in $\S(\T_\Omega)$ with respect to the $L^2(\Omega)$-norm. By standard interpolation estimates (see e.g. 
%\begin{align}
%\label{L^2Estimate1}
%\|U-\Re_U\|_{L^2(\Omega)}=\inf_{s\in\S(\T_\Omega)}\|U-s\|_{L^2(\Omega)}\le %Ch_\Omega^2\|U\|_{H^2(\Omega)}.
%\end{align}
%Observe that the constant $C$ is independent of time (although of course $\|U\|_{H^2(\Omega)}$ depends on $t$.) By (\ref{GalStabil1}) and (\ref{UerrorLess}), we have
%\begin{eqnarray}
%\nonumber
%\|U-\Re_U\|_{H^k(\Omega)}&\le&\|U-\tU\|_{H^k(\Omega)}+\|e_U^1\|_{H^k(\Omega)}\le %C\|U-\tU\|_{H^k(\Omega)}\\
%\nonumber
%&\le& C\|U\|_{H^k(\Omega)} 
%\end{eqnarray}
%The previous inequality for $k=0,2$ together with (\ref{gagliardo}) and (\ref{L^2Estimate1}) give
%\begin{eqnarray}
%\nonumber
%\|\nabla(U-\Re_U)\|_{L^2(\Omega)}&\le& %C\|U-\Re_U\|^{1/2}_{H^2(\Omega)}\|U-\Re_U\|^{1/2}_{L^2(\Omega)}\\
%\nonumber
%&\le& Ch_\Omega\|U\|^{1/2}_{H^2(\Omega)}\|U-\Re_U\|^{1/2}_{L^2(\Omega)}\\
%\nonumber
%&\le&Ch_\Omega\|U\|_{H^2(\Omega)}.
%\end{eqnarray}
%Integrating over $S$ gives (\ref{aprioriU1}).
Due to the tensor product structure of our spaces, the proof for $\Re_v,\Re_w$ is very similar; we sketch the proof for $\Re_v$. Since $v-\Re_v=e_v^2\in L^2(\Omega)\times\S(\T_Y)^\perp$, we have that $v-\Re_v$ is orthogonal to every function in $L^2(\Omega)\times\S(\T_Y)$. Thus, for every $\varphi\in L^2(\Omega)$, every $s\in\S(\T_Y)$ and a.e. $t\in S$ we have
\begin{align}
\nonumber
\langle v-\Re_v,\varphi s\rangle_{L^2(\Omega,H^1(Y))}=0.
\end{align}
It easily follows that for a.e. $x\in\Omega$ and a.e. $t\in S$
\begin{align}
\nonumber
\langle v(t,x,\cdot)-\Re_v(t,x,\cdot),s\rangle_{H^1(Y)}=0\quad{\rm for\,\,all}\quad s\in\S(\T_Y),
\end{align}
and we again use estimates on elliptic projections to obtain
\begin{align}
\nonumber
\|v(t,x,\cdot)-\Re_v(t,x,\cdot)\|_{H^k(Y)}\le Ch^{2-k}_Y\|v(t,x,\cdot)\|_{H^2(Y)}\quad(k=0,1),
\end{align}
where $C$ is independent of $t,x$. Integrating over $S\times\Omega$ yields (\ref{apriorivw1}).
\end{proof}

\subsection{Error of Galerkin approximation}
We need the following result on continuity with respect to data.
\begin{prop}
\label{continuityData}
Consider the $x$-dependent auxiliary system posed in $\Omega\times Y$
\begin{equation}
\label{contSystem}
\left\{\begin{array}{rcl}
\partial_t\ov-D_v\Delta_y\ov&=&-\eta(\ov,\ow)\\
\partial_t\ow-D_w\Delta_y\ow&=&-\eta(\ov,\ow)
\end{array} \right.
\end{equation}
with boundary conditions
\begin{equation}
\label{contBoundary}
\left\{\begin{array}{rcl}
-D_v\nabla_y\ov\cdot n_y&=&\a(\ov-\oU)\quad{\rm at}\quad S\times\Omega\times\Gamma^R\\
-D_v\nabla_y\ov\cdot n_y&=&0\quad{\rm at}\quad S\times\Omega\times\Gamma\setminus\Gamma^R\\
-D_w\nabla_y\ow\cdot n_y&=&0\quad{\rm at}\quad S\times\Omega\times\Gamma
\end{array} \right.
\end{equation}
and initial conditions
\begin{align}
\nonumber 
\ov(0,x,y)=\ov_I(x,y),\quad \ow(0,x,y)=\ow_I(x,y),\quad(x,y)\in\Omega\times Y
\end{align}
Assume that $\oU^1,\oU^2\in L^2(S,H^1(\Omega))$ and $(\ov^1,\ow^1),(\ov^2,\ow^2)$ are solutions to (\ref{contSystem}) with data $\oU^1,\oU^2$ respectively and the same initial conditions
\begin{align}
\nonumber 
\ov^1(0,x,y)=\ov^2(0,x,y),\quad \ow^1(0,x,y)=\ow^2(0,x,y).
\end{align}
Assume that $\ov_I,\ow_I\in L^2(\Omega,H^1(Y))$, then 
\begin{align}
\label{contEstimate1}
\|\ov^2-\ov^1\|^2_\cX +\|\ow^2-\ow^1\|^2_\cX \le C\|\oU^2-\oU^1\|^2_{L^2(S,L^2(\Omega))}
\end{align}
where $\cX=L^2(S,L^2(\Omega,H^1(Y)))$.
\end{prop}
\begin{proof}
The weak form of (\ref{contSystem}) is
\begin{align}
\nonumber
\int_{\Omega\times Y}\partial_t\ov\varphi+D_v\int_{\Omega\times Y}\nabla_y\ov\nabla_y\varphi-\a\int_{\Omega\times\Gamma^R}(\ov-\tU)\varphi d\sigma_ydx\\
\label{contWeakV}
=\int_{\Omega\times Y}\eta(\ov,\ow)\varphi\\
\label{contWeakW}
\int_{\Omega\times Y}\partial_t\ow\phi+D_w\int_{\Omega\times Y}\nabla_y\ow\nabla_y\phi
=\int_{\Omega\times Y}\eta(\ov,\ow)\phi
\end{align}
for any $(\varphi,\phi)\in L^2(\Omega,H^1(Y))^2$. Consider (\ref{contWeakV}) and (\ref{contWeakW}) for $\oU_1,\oU_2$ and set
\begin{align}
\nonumber
e_{\ov}:=\ov^2-\ov^1,\quad e_{\ow}:=\ow^2-\ow^1,\quad e_{\oU}:=\oU^2-\oU^1.
\end{align}
Subtracting the weak formulations, we obtain
\begin{align}
\nonumber
\int_{\Omega\times Y}\partial_t\ev\varphi+D_v\int_{\Omega\times Y}\nabla_y\ev\nabla_y\varphi-\a\int_{\Omega\times\Gamma^R}(\ev-\eu)\varphi d\sigma_ydx\\
\nonumber
=\int_{\Omega\times Y}[\eta(\ov^2,\ow^2)-\eta(\ov^1,\ow^1)]\varphi\\
\nonumber
\int_{\Omega\times Y}\partial_t\ew\phi+D_w\int_{\Omega\times Y}\nabla_y\ew\nabla_y\phi
=\int_{\Omega\times Y}[\eta(\ov^2,\ow^2)-\eta(\ov^1,\ow^1)]\phi.
\end{align}
Testing the previous equations with $(\varphi,\phi)=(\ev,\ew)$ yields
\begin{eqnarray}
\nonumber
\frac{d}{2dt}\left[\|\ev\|^2_{L^2(\Omega\times Y)}\right.&+&\left.\|\ew\|^2_{L^2(\Omega\times Y)}\right]+\\
\nonumber
+ D_v\|\nabla_y\ev\|^2_{L^2(\Omega\times Y)}&+& D_w\|\nabla_y\ew\|^2_{L^2(\Omega\times Y)}=\\
\label{contEstim2}
=\a\int_{\Omega\times\Gamma^R}(\ev-\eu)\ev d\sigma_ydx
&+&\int_{\Omega\times Y}\delta(\eta)(\ev+\ew),
\end{eqnarray}
where $\delta(\eta)=\eta(\ov^2,\ow^2)-\eta(\ov^1,\ow^1)$. We proceed to estimate the right-hand side of (\ref{contEstim2}). 

Using (\ref{InterpolationTrace}) with parameter $\rho>0$, we get
\begin{align}
\nonumber
\int_{\Omega\times\Gamma^R}|(\ev-\eu)\ev|d\sigma_ydx\le \int_{\Omega\times\Gamma^R}\left[(\ev-\eu)^2+\ev^2\right]d\sigma_ydx\\
\label{Cont3}
\le\rho\int_{\Omega\times Y}\nabla_y\ev^2+c_\rho\|\ev\|^2_{L^2(\Omega\times Y)}+c\|\eu\|^2_{L^2(\Omega)}.
\end{align}
Using the Lipschitz continuity of $\eta$, we have
\begin{align}
\label{Cont4}
|\delta(\eta)|\le C(|\ev|+|\ew|).
\end{align}
Hence, it holds
\begin{align}
\nonumber
\int_{\Omega\times Y}|\delta(\eta)(\ev+\ew)|\le C\left[\|\ev\|^2_{L^2(\Omega\times Y)}+\|\ew\|^2_{L^2(\Omega\times Y)}\right].
\end{align}
Putting (\ref{Cont3}) and (\ref{Cont4}) into (\ref{contEstim2}) and rearranging the terms, we obtain
\begin{eqnarray}
\nonumber
\frac{d}{2dt}\left[\|\ev\|^2_{L^2(\Omega\times Y)}\right.&+&\left.\|\ew\|^2_{L^2(\Omega\times Y)}\right]+\\
\nonumber
+(D_v-\rho)\|\nabla_y\ev\|^2_{L^2(\Omega\times Y)}&+& D_w\|\nabla_y\ew\|^2_{L^2(\Omega\times Y)}\le\\
\label{contEstim5}
\le C\left[\|\eu\|^2_{L^2(\Omega)}\right.&+&\left.\|\ev\|^2_{L^2(\Omega\times Y)}+\|\ew\|^2_{L^2(\Omega\times Y)}\right]
\end{eqnarray}
Chose $\rho=D_v/2$ in (\ref{contEstim5}) and define
\begin{align}
\nonumber
\Xi(t):=\|\ev\|^2_{L^2(\Omega\times Y)}+\|\ew\|^2_{L^2(\Omega\times Y)}
\end{align}
and
\begin{align}
\nonumber
\mu(t)=D_v\|\nabla_y\ev\|^2_{L^2(\Omega\times Y)}+D_w\|\nabla_y\ew\|^2_{L^2(\Omega\times Y)},
\end{align}
then (\ref{contEstim5}) can be written as
\begin{align}
\nonumber
\Xi'(t)+\mu(t)\le C\Xi(t)+C\|\eu(t)\|^2_{L^2(\Omega)}.
\end{align}
Since $\mu(t)>0$, Gr\"{o}nwall's inequality and the fact that $\Xi(0)=0$ implies that
\begin{align}
\nonumber
\Xi(t)\le C\int_0^T\|\eu(s)\|^2_{L^2(\Omega)}ds.
\end{align}
Further, we get
\begin{align}
\nonumber
\mu(t)\le C\int_0^T\|\eu(s)\|^2_{L^2(\Omega)}ds+C\|\eu(t)\|^2_{L^2(\Omega)}.
\end{align}
Using the above observations, we finally obtain
\begin{align}
\nonumber
\|\ev\|^2_\cX +\|\ew\|^2_\cX\le C\int_0^T[\Xi(t)+\mu(t)]dt\le C\int_0^T\|\eu(t)\|^2_{L^2(\Omega)}dt,
\end{align}
which concludes the proof.
\end{proof}
\begin{rem}
\label{contExistence}
Existence of weak/strong solutions to (\ref{contSystem}) is obtained in the same way as for (\ref{system}), see \cite{MunteanNeussRadu}, while uniqueness is a direct outcome of (\ref{contEstimate1}).
\end{rem}

The next theorem states that the microscopic errors can be estimated by the macroscopic error.
\begin{teo}
\label{micromacroError} 
Given partitions $\T_\Omega,\T_Y$, there exists an absolute constant $c^\star$ such that
\begin{align}
\label{mainErrorEstimate}
\|e_v\|^2_\cX+\|e_w\|^2_\cX\leq c^\star\|e_U\|_{L^2(S,L^2(\Omega))}^2+\O(h^2_Y).
\end{align}
\end{teo}
\begin{rem}
Note that $\|v-\Re_v\|^2_\cX+\|w-\Re_w\|^2_\cX=\O(h_Y^2)$ does not imply $\|e_v\|^2_\cX+\|e_w\|^2_\cX=\O(h_Y^2)$ immediately. Indeed, $\Re_v\neq\tv, \Re_w\neq\tw$.
\end{rem}
\begin{proof}[Proof of Proposition \ref{micromacroError}]

Let $\oU^2=U$ be the weak macroscopic solution of (\ref{system}) and $\oU^1=\tU$ be the Galerkin projection subordinate to $\T_\Omega$. Solving (\ref{contSystem}) with data $\oU^1$ and $\oU^2$ and initial conditions $(v_I,w_I)$ in both cases, we obtain solutions $(\ov^1,\ow^1)$ and $(\ov^2,\ow^2)$ respectively. By uniqueness of the weak solution to (\ref{contSystem}) (see Remark \ref{contExistence} above), we have $(\ov^2,\ow^2)=(v,w)$.
Note also that in general $(\ov^1,\ow^1)\neq(\tv,\tw)$. 

By Proposition \ref{continuityData}
\begin{align}
\nonumber
\|v-\ov^1\|_\cX^2+\|w-\ow^1\|_\cX^2\le C\|U-\tU\|^2_{L^2(S,L^2(\Omega))}.
\end{align}
Hence, 
\begin{align}
\nonumber
\|v-\tv\|_\cX^2+\|w-\tw\|_\cX^2\le\|v-\ov^1\|_\cX^2+\|\ov^1-\tv\|_\cX^2+\|w-\ow^1\|_\cX^2\\
\nonumber
+\|\ow^1-\tw\|_\cX^2
\le\|\ov^1-\tv\|_\cX^2+\|\ow^1-\tw\|_\cX^2+C\|U-\tU\|_{L^2(S,L^2(\Omega))}^2.
\end{align}
We must estimate $\|\ov^1-\tv\|_\cX, \|\ow^1-\tw\|_\cX$. 
For all $(\psi,\phi)\in [\S(\T_\Omega)\times\S(\T_Y)]^2$, we have
\begin{align}
\nonumber
\int_{\Omega\times Y}\partial_t(\ov^1-\tv)\psi+D_v\int_{\Omega\times Y}\nabla_y(\ov^1-\tv)\nabla_y\psi-\a\gamma\int_{\Omega\times\Gamma^R}(\ov^1-\tv)\psi d\sigma_ydx\\
\label{mainErrorEstimate2}
=\int_{\Omega\times Y}[\eta(\ov^1,\ow^1)-\eta(\tv,\tw)]\psi.
\end{align}
and
\begin{align}
\nonumber
\int_{\Omega\times Y}\partial_t(\ow^1-\tw)\phi+D_w\int_{\Omega\times Y}\nabla_y(\ow^1-\tw)\nabla_y\phi\\
\label{mainErrorEstimate3}
=\int_{\Omega\times Y}[\eta(\ov^1,\ow^1)-\eta(\tv,\tw)]\phi.
\end{align}
Let $\Re_{\ov^1},\Re_{\ow^1}\in\S(\T_\Omega)\times\S(\T_Y)$ be the functions given by Proposition \ref{RieszProjections} associated to $\ov^1,\ow^1$. We shall test the equations (\ref{mainErrorEstimate2}) and (\ref{mainErrorEstimate3}) with $(\psi,\phi)=(\Re_{\ov^1}-\tv,\Re_{\ow^1}-\tw)$. Note that we have
\begin{align}
\nonumber
\int_{\Omega\times Y}\partial_t(\ov^1-\tv)(\Re_{\ov^1}-\tv)=\frac{d}{2dt}\|\ov^1-\tv\|^2_{L^2(\Omega\times Y)}\\
\nonumber
-\int_{\Omega\times Y}\partial_t(\ov^1-\tv)(\ov^1-\Re_{\ov^1}),
\end{align}
\begin{align}
\nonumber
\int_{\Omega\times Y}\nabla_y(\ov^1-\tv)\nabla_y(\Re_{\ov^1}-\ov)=\|\nabla_y(\ov^1-\tv)\|^2_{L^2(\Omega\times Y)}\\
\nonumber
-\int_{\Omega\times Y}\nabla_y(\ov^1-\tv)\nabla_y(\ov^1-\Re_{\ov^1}).
\end{align}
and the same for the equation related to $\ow^1$. 
By adding the equations and rearranging, we obtain
\begin{align}
\nonumber
\frac{d}{2dt}\left(\|\ov^1-\tv\|^2_{L^2(\Omega\times Y)}+\|\ow^1-\tw\|^2_{L^2(\Omega\times Y)}\right)+D_v\|\nabla_y(\ov^1-\tv)\|^2_{L^2(\Omega\times Y)}\hspace{1.0cm} \\
\nonumber
+D_w\|\nabla_y(\ow^1-\tw)\|^2_{L^2(\Omega\times Y)}\le\int_{\Omega\times Y}|\partial_t(\ov^1-\tv)||\ov^1-\Re_{\ov^1}|\hspace{2.0cm}\\
\nonumber
+\int_{\Omega\times Y}|\partial_t(\ow^1-\tw)||\ow^1-\Re_{\ow^1}|
+D_v\int_{\Omega\times Y}|\nabla_y(\ov^1-\tv)||\nabla_y(\ov^1-\Re_{\ov^1})|\hspace{1.0cm}\\
\nonumber
+D_w\int_{\Omega\times Y}|\nabla_y(\ow^1-\tw)||\nabla_y(\ow^1-\Re_{\ow^1})|
+\a\gamma\int_{\Omega\times\Gamma^R}|\ov^1-\tv||\Re_{\ov^1}-\tv| d\sigma_ydx\\
\nonumber
+\int_{\Omega\times Y}|\eta(\ov^1,\ow^1)-\eta(\tv,\tw)|(|\Re_{\ov^1}-\tv|+|\Re_{\ow^1}-\tw|)\hspace{3.0cm}\\
\label{estim0}
=I_1+I_2+I_3+I_4+I_5+I_6\hspace{7.0cm}
\end{align}
We proceed to estimate $I_1,...,I_6$. Clearly we have
\begin{align}
\label{estim1}
I_1\le\|\partial_t(\ov^1-\tv)\|_{L^2(\Omega\times Y)}\|\ov^1-\Re_{\ov^1}\|_{L^2(\Omega\times Y)}
\end{align}
and
\begin{align}
\label{estim12}
I_2\le\|\partial_t(\ow^1-\tw)\|_{L^2(\Omega\times Y)}\|\ow^1-\Re_{\ow^1}\|_{L^2(\Omega\times Y)}.
\end{align}
Using (\ref{epsilonAMGM}) with parameter $\rho>0$ and (\ref{apriorivw1}), we get
\begin{align}
\nonumber
I_3=D_v\int_{\Omega\times Y}|\nabla_y(\ov^1-\tv)||\nabla_y(\ov^1-\Re_{\ov^1})|\le\\
\label{estim2}
\le D_v\rho\|\nabla_y(\ov^1-\tv)\|^2_{L^2(\Omega\times Y)}+c_\rho\|\nabla_y(\ov^1-\Re_{\ov^1})\|^2_{L^2(\Omega\times Y)}.
\end{align}
By the same token,
\begin{align}
\nonumber
I_4=D_w\int_{\Omega\times Y}|\nabla_y(\ow^1-\tw)||\nabla_y(\ow^1-\Re_{\ow^1})|\le\\
\label{estim22}
\le D_w\rho\|\nabla_y(\ow^1-\tw)\|^2_{L^2(\Omega\times Y)}+c_\rho\|\nabla_y(\ow^1-\Re_{\ow^1})\|^2_{L^2(\Omega\times Y)}
\end{align}
Further, (\ref{InterpolationTrace}) with parameter $\rho>0$ leads to the following estimates
\begin{align}
\nonumber
I_5=\a\gamma\int_{\Omega\times\Gamma^R}|\ov^1-\tv||\Re_{\ov^1}-\tv| d\sigma_ydx\hspace{4.0cm}\\
\nonumber
\le\a\gamma\int_{\Omega\times\Gamma^R}(\ov^1-\tv)^2d\sigma_ydx+\a\gamma\int_{\Omega\times\Gamma^R}|\ov^1-\tv||\ov^1-\Re_{\ov^1}|d\sigma_ydx\\
\nonumber
\le\frac{3\a\gamma}{2}\int_{\Omega\times\Gamma^R}\left[(\ov^1-\tv)^2+(\ov^1-\Re_{\ov^1})^2\right]d\sigma_ydx\hspace{3.0cm}\\
\label{estim3}
\le C\left[\rho\|\nabla_y(\ov^1-\tv)\|^2_{L^2(\Omega\times Y)}+\|\ov^1-\tv\|^2_{L^2(\Omega\times Y)}+\|\ov^1-\Re_{\ov^1}\|^2_{L^2(\Omega, H^1(Y))}\right]
\end{align}
Finally, using the Lipschitz continuity of $\eta$, it is not difficult to see that
\begin{align}
\nonumber
I_6\le C\left(\|\ov^1-\tv\|^2_{L^2(\Omega\times Y)}+\|\ow^1-\tw\|^2_{L^2(\Omega\times Y)}\right)+\\
\label{estim4}
+C\left(\|\ov^1-\Re_{\ov^1}\|^2_{L^2(\Omega\times Y)}+\|\ow^1-\Re_{\ow^1}\|^2_{L^2(\Omega\times Y)}\right).
\end{align}
Set
\begin{align}
\nonumber
\Xi(t):=\|\ov^1-\tv\|^2_{L^2(\Omega\times Y)}+\|\ow^1-\tw\|^2_{L^2(\Omega\times Y)},
\end{align}
and
\begin{align}
\nonumber
\Upsilon(t):=D_v\|\nabla_y(\ov^1-\tv)\|^2_{L^2(\Omega\times Y)}+D_w\|\nabla_y(\ow^1-\tw)\|^2_{L^2(\Omega\times Y)}.
\end{align}
Taking all estimates (\ref{estim0})-(\ref{estim4}) into consideration, we obtain the estimate
\begin{align}
\nonumber
\Xi'(t)+2\Upsilon(t)\le C_1\Xi(t)+I_1+I_2+C_2\rho\Upsilon(t)+\\
\nonumber
+C_3\left(\|\ov^1-\Re_{\ov^1}\|^2_{L^2(\Omega,H^1(Y))}+\|\ow^1-\Re_{\ow^1}\|^2_{L^2(\Omega,H^1(Y))}\right).
\end{align}
Choosing $\rho=1/C_2$ and denoting
\begin{align}
\nonumber
\mu(t):=I_1+I_2+C_3\left(\|\ov^1-\Re_{\ov^1}\|^2_{L^2(\Omega,H^1(Y))}+\|\ow^1-\Re_{\ow^1}\|^2_{L^2(\Omega,H^1(Y))}\right),
\end{align}
we obtain
\begin{align}
\nonumber
\Xi'(t)+\Upsilon(t)\le C_1\Xi(t)+\mu(t).
\end{align}
By Gr\"{o}nwall's inequality, we get
\begin{align}
\nonumber
\Xi(t)\le C\left(\Xi(0)+\int_0^T\mu(s)ds\right).
\end{align}
By (\ref{L^2TimeDerivative}) below, we have
\begin{align}
\nonumber
\|\partial_t(\ov^1-\tv)\|_{L^2(S,L^2(\Omega\times Y))}+\|\partial_t(\ow^1-\tw)\|_{L^2(S,L^2(\Omega\times Y))}\le C.
\end{align}
This together with (\ref{apriorivw1}) yields
\begin{align}
\nonumber
\int_0^T (I_1+I_2)ds\le\|\partial_t(\ov^1-\tv)\|_{L^2(S,L^2(\Omega\times Y))}\|\ov^1-\Re_{\ov^1}\|_{L^2(S,L^2(\Omega\times Y))}\\
\nonumber
+\|\partial_t(\ow^1-\tw)\|_{L^2(S,L^2(\Omega\times Y))}\|\ow^1-\Re_{\ow^1}\|_{L^2(S,L^2(\Omega\times Y))}\le Ch_Y^2.
\end{align}
Thus, we are lead to
\begin{align}
\nonumber
\int_0^T\mu(s)ds\le Ch_Y^2,
\end{align}
and we obtain
\begin{align}
\nonumber
\Xi(t)\le C\left[\Xi(0)+h_Y^2\right].
\end{align}
Using the previous two inequalities, we also get
\begin{align}
\nonumber
\int_0^T\Upsilon(t)dt\le C\int_0^T\Xi(t)+\mu(t)dt\le C\Xi(0)+Ch_Y^2.
\end{align}
Consequently,
\begin{align}
\nonumber
\|\ov^1-\tv\|^2_\cX+\|\ow^1-\tw\|^2_\cX\le C\Xi(0)+Ch_Y^2.
\end{align}
Now, since $\ov^1(0,x,y)=v(0,x,y)=v_I(x,y)$ and $\ow^1(0,x,y)=w(0,x,y)=w_I(x,y)$, it follows that $\Xi(0)$ can be incorporated in the $\O(h_Y^2)$ term since we assume $\tv,\tw$ approximate $v_I,w_I$ well. 
\end{proof}

\begin{rem}
Our final result in this section shows that we can control the error in $U$ by its projection onto the orthogonal complement $\S(\T_\Omega)^\perp$.
\end{rem}
\begin{prop}
\label{UErrorProp} 
Let $\T_\Omega,\T_Y$ be arbitrary partitions and set
\begin{align}
\nonumber
\varepsilon_U:=\max\left\{\|e^2_U\|_{L^2(S,H^1(\Omega))},\|e^2_U\|^2_{L^2(S,H^1(\Omega))}\right\}.
\end{align}
Then
\begin{align}
\label{errLimit0}
\|e_U\|_{L^2(S,H^1(\Omega))}\le C_U\sqrt{\varepsilon_U}+C\|e_U^1(0)\|_{L^2(\Omega)}.
\end{align}
\end{prop}
\begin{proof}
Subtracting (\ref{GalU}) from (\ref{weakU}) we obtain
\begin{align}
\nonumber
\int_\Omega\partial_te_U\varphi+D_U\int_\Omega\nabla e_U\nabla\varphi
=\gamma\a\int_{\Omega\times\Gamma^R}(e_v-e_U)\varphi d\sigma_ydx
\end{align}
for all $\varphi\in\S(\T_\Omega)$. Taking $\varphi=e_U^1=U-\tU-e_U^2$ gives
\begin{align}
\nonumber
\int_\Omega\partial_te_Ue_U^1+D_U\int_\Omega\nabla e_U\nabla(e_U-e_U^2)=\frac{d}{2dt}\|e^1_U\|^2_{L^2(\Omega)}\\
+\int_\Omega\partial_te^1_Ue_U^2+D_U\|\nabla e_U\|^2_{L^2(\Omega)}-D_U\|e_U^2\|^2_{L^2(\Omega)},
\end{align}
where we used that $\partial_te_U^1\in\S(\T_\Omega)$ and $e_U^2\in\S(\T_\Omega)^\perp$.
Hence,
\begin{align}
\nonumber
\frac{d}{dt}\|e^1_U\|^2_{L^2(\Omega)}+2D_U\|\nabla e_U\|^2_{L^2(\Omega)}\le2\int_{\Omega}|\partial_te_U^1e_U^2|\\
\label{UErrorProp1}
+2D_U\|\nabla e_U^2|^2_{L^2(\Omega)}+2\gamma\a\int_{\Omega\times\Gamma^R}|e_v-e_U||e_U^1|d\sigma_ydx.
\end{align}
For the first term of (\ref{UErrorProp1}), we have 
\begin{align}
\label{UErrorProp2}
\int_{\Omega}|\partial_te_U^1e_U^2|\le\|\partial_te_U^1\|_{L^2(\Omega)}\|e_U^2\|_{L^2(\Omega)}
\end{align}
Further, the second term of (\ref{UErrorProp1}) can be estimated by applying again (\ref{epsilonAMGM}) with parameter $\rho>0$ and (\ref{InterpolationTrace})
\begin{align}
\nonumber
\int_{\Omega\times\Gamma^R}|e_v-e_U||e_U^1|d\sigma_ydx\le \rho\int_{\Omega\times\Gamma^R}(e_v-e_U)d\sigma_ydx+c_\rho\|e_U^1\|^2_{L^2(\Omega)}\\
\nonumber
\le C_1\rho\|e_v\|^2_{L^2(\Omega,H^1(Y))}+C\left(\|e_U^1\|^2_{L^2(\Omega)}+\|e_U\|^2_{L^2(\Omega)}\right)\\
\label{UErrorProp3}
\le C_1\rho\|e_v\|^2_{L^2(\Omega,H^1(Y))}+C\left(\|e_U^1\|^2_{L^2(\Omega)}+\|e^2_U\|^2_{L^2(\Omega)}\right),
\end{align}
where we used the inequality $\|e_U\|_{L^2(\Omega)}\le\|e_U^1\|_{L^2(\Omega)}+\|e^2_U\|_{L^2(\Omega)}$

By (\ref{UErrorProp2}) and (\ref{UErrorProp3}) we rearrange (\ref{UErrorProp1}) to
\begin{align}
\nonumber
\frac{d}{dt}\|e^1_U\|^2_{L^2(\Omega)}+D_U\|\nabla e_U\|^2_{L^2(\Omega)}\le C\|e_U^1\|^2_{L^2(\Omega)}\\
\nonumber
+C_1\rho\|e_v\|^2_{L^2(\Omega,H^1(Y))}+C\|e_U^2\|^2_{H^1(\Omega)}+\|\partial_te_U^1\|_{L^2(\Omega)}\|e_U^2\|_{L^2(\Omega)}.
\end{align}
Set
\begin{align}
\nonumber
\mu(t):=C_1\rho\|e_v(t)\|^2_{L^2(\Omega,H^1(Y))}+\|e_U^2(t)\|^2_{H^1(\Omega)}+\|\partial_te_U^1(t)\|_{L^2(\Omega)}\|e_U^2(t)\|_{L^2(\Omega)},
\end{align}
then
\begin{align}
\nonumber
\frac{d}{dt}\|e^1_U\|^2_{L^2(\Omega)}\le C\|e_U^1(t)\|^2_{L^2(\Omega)}+\mu(t),
\end{align}
and, by Gr\"{o}nwall's inequality, we have for all $t\in[0,T]$
\begin{align}
\label{UErrorProp4}
\|e_U^1(t)\|^2_{L^2(\Omega)}\le C\left(\|e^1_U(0)\|^2_{L^2(\Omega)}+\int_0^T\mu(s)ds\right).
\end{align}
From (\ref{UErrorProp4}), we first get 
\begin{align}
\nonumber
D_U\|\nabla e_U(t)\|^2_{L^2(\Omega)}\le C\|e_U^1(t)\|^2_{L^2(\Omega)}+\mu(t)\\
\label{UErrorProp5}
\le\mu(t)+C\left(\|e^1_U(0)\|^2_{L^2(\Omega)}+\int_0^T\mu(s)ds\right).
\end{align}
It also follows from (\ref{UErrorProp4}) that
\begin{align}
\nonumber
\|e_U(t)\|^2_{L^2(\Omega)}\le\|e_U^1(t)\|^2_{L^2(\Omega)}+\|e_U^2(t)\|^2_{L^2(\Omega)}\\
\label{UErrorProp6}
\le \|e_U^2(t)\|^2_{L^2(\Omega)}+C\left(\|e^1_U(0)\|^2_{L^2(\Omega)}+\int_0^T\mu(s)ds\right).
\end{align}
Adding (\ref{UErrorProp5}) and  (\ref{UErrorProp6}) integrating over $[0,T]$ yields
\begin{align}
\nonumber
\|e_U\|^2_{L^2(S,H^1(\Omega))}\le C\|e_U^1(0)\|^2_{L^2(\Omega)}+C\|e_U^2\|^2_{L^2(S,L^2(\Omega))}\\
C\int_0^T\mu(t)dt\le
C\|e_U^1(0)\|^2_{L^2(\Omega)}+C\|e_U^2\|^2_{L^2(S,H^1(\Omega))}
\\
\label{UErrorProp7}
+C\rho\|e_v\|^2_{L^2(S,L^2(\Omega.H^1(Y)))}
+C\int_0^T\|\partial_te_U^1(t)\|_{L^2(\Omega)}\|e_U^2(t)\|_{L^2(\Omega)}dt.
\end{align}
By (\ref{mainErrorEstimate}), we may choose $\rho$ sufficiently small to have 
\begin{align}
\nonumber
\rho\|e_v\|^2_{L^2(S,L^2(\Omega.H^1(Y)))}\le\frac{1}{2}\|e_U\|^2_{L^2(S,H^1(\Omega))},
\end{align}
then we get from (\ref{UErrorProp7})
\begin{align}
\nonumber
\|e_U\|^2_{L^2(S,H^1(\Omega))}\le C\|e_U^1(0)\|^2_{L^2(\Omega)}+C\|e_U^2\|^2_{L^2(S,L^2(\Omega))}+\\
\nonumber
+C\int_0^T\|\partial_te_U^1(t)\|_{L^2(\Omega)}\|e_U^2(t)\|_{L^2(\Omega)}dt\\
\nonumber
\le C\|e_U^1(0)\|^2_{L^2(\Omega)}+C\|e_U^2\|^2_{L^2(S,L^2(\Omega))}+\|\partial_te_U^1\|_{L^2(S,L^2(\Omega))}\|e_U^2\|_{L^2(S,L^2(\Omega))}\\
\nonumber
\le C\|e_U^1(0)\|^2_{L^2(\Omega)}+C_U\max\{\|e_U^2\|_{L^2(S,L^2(\Omega))},\|e_U^2\|^2_{L^2(S,L^2(\Omega))}\}
\end{align}
where
\begin{align}
\nonumber
C_U=C[1+\|\partial_te_U^1\|_{L^2(S,L^2(\Omega))}].
\end{align}
This concludes the proof.
\end{proof}

\section{Feedback convergence}

\subsection{Dyadic partitions}
For the sake of simplicity, we assume in this section that $\Omega=[0,1]^2$. A \emph{dyadic square} in $\Omega$ is a set of the form 
\begin{align}
\nonumber
[i2^{-m},(i+1)2^{-m}]\times[j2^{-m},(j+1)2^{-m}],\quad (m\in\N,\,\, 0\le i,j\le2^{m}-1).
\end{align}
Given any dyadic square $Q$, we denote by $\mathcal{C}(Q)$ its four children obtained by bisecting each side of $Q$.

A \emph{dyadic partition} $\T=\{Q\}$ is a finite set of nonoverlapping dyadic squares such that
\begin{align}
\nonumber
\Omega=\bigcup_{Q\in\T}Q.
\end{align}
See Figure \ref{figDyad} for a sketch of a dyadic partition.
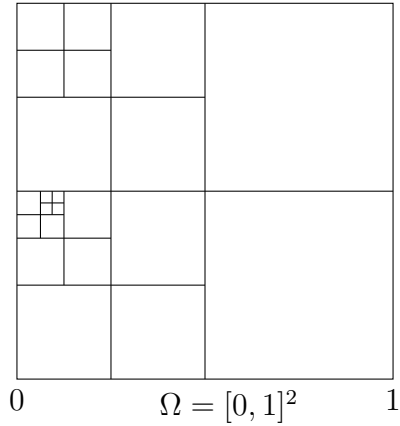
\begin{figure}[H]
\begin{tikzpicture}[scale=5]
\clip(-0.1,-0.15) rectangle (1.1,1.1);
\draw (0,0) -- (1,0) -- (1,1) -- (0,1) -- cycle;
\draw (0,0.5)-- (1,0.5);
\draw (0.5,0)-- (0.5,1);
\draw (0,0.75)-- (0.5,0.75);
\draw (0.25,0)-- (0.25,1);
\draw (0.125,0.75)-- (0.125,1);
\draw (0,0.875)-- (0.25,0.875);
\draw (0,0.375)-- (0.25,0.375);
\draw (0.125,0.25)-- (0.125,0.5);
\draw (0,0.25)-- (0.5,0.25);
\draw (0.0625,0.375)-- (0.0625,0.5);
\draw (0,0.4375)-- (0.125,0.4375);
\draw (0.09375,0.5)-- (0.0935,0.4375);
\draw (0.0625,0.46875)--(0.125,0.46875);
\draw (0.35,0) node[anchor=north west] {$\Omega=[0,1]^2$};
\draw (-0.05,0) node[anchor=north west] {$0$};
\draw (0.95,0) node[anchor=north west] {$1$};
\end{tikzpicture}
\caption{\label{figDyad} A dyadic partition.}
\end{figure}

\subsection{The refinement scheme}
Let $\T=\{Q\}$ be a dyadic partition of $\Omega$ and let $\T_Y$ be an arbitrary partition of $Y$. Let $(\tU,\tv,\tw)$ be the Galerkin projections of the weak solution $(U,v,w)$. Recall that 
\begin{align}
\nonumber
U-\tU=e_U=e_{U,\T}^1+e_{U,\T}^2,\quad{\rm where}\quad e_{U,\T}^1\in\S(\T),\,e_{U,\T}^2\in\S(\T)^\perp.
\end{align}
We define the \emph{error indicator} at $Q\in\T$ as
\begin{align}
\label{errInd}
\nu_\T(Q)=\left(\int_0^T\int_Q((e_{U,\T}^2)^2+(\nabla e_{U,N}^2)^2)dxdt\right)^{1/2}.
\end{align}
We shall now describe a way of generating a sequence $\{\T_i\}$ of dyadic partitions via feedback. Our discussion follows \cite{BabuskaVogelius} closely.

Assume that at some stage we have $\T_1,\T_2,...,\T_i$. To generate $\T_{i+1}$ from $\T_i=\{Q_k\}$, we set
\begin{align}
\nonumber
J_i=\left\{k:\nu_{\T_i}(Q_k)\ge\beta\max_{Q\in\T_i}\nu_{\T_i}(Q)\right\}
\end{align}
for some $\b\in(0,1)$. In other words, we mark the squares of $\T_i$ where the error indicator (\ref{errInd}) is large. Note that at least one square will always be marked.
Clearly the effect of the refinement scheme is to equidistribute the error.

The marked squares $\{Q_k:k\in J_i\}$ are subdivided to give $\T_{i+1}$, i.e.
\begin{align}
\label{refinement}
\T_{i+1}=\{\mathcal{C}(Q_k):k\in J_i\}\cup\{Q_k\in\T_i:k\notin J_i\},
\end{align}
where $\mathcal{C}(Q)$ denotes the children of $Q$.

\begin{prop}
\label{convergenceProp}
Let $\{\T_N\}$ be a sequence of dyadic partitions obtained from the refinement scheme and let
\begin{align}
\nonumber
e_{U,N}^2=e_{U,\T_N}^2.
\end{align}
Then we have
\begin{align}
\label{errLimit}
\lim_{N\rightarrow\infty}\|e_{U,N}^2\|_{L^2(S,H^1(\Omega))}=0.
\end{align}
\end{prop}
\begin{proof}

Since $\T_{i+1}$ is a refinement of $\T_i$, we have
\begin{align}
\nonumber
\S(\T_1)\subset\S(\T_2)\subset...\subset\S(\T_i)\subset\S(\T_{i+1})\subset...\subset H^1(\Omega)
\end{align}
Consider on $L^2(S,H^1(\Omega))$ the inner product 
\begin{align}
\nonumber
\langle \varphi,\psi\rangle_{L^2(S,H^1(\Omega))}=\int_0^T\int_{\Omega}\nabla \varphi\nabla \psi+\varphi\psi dxdt,\quad  \varphi,\psi\in L^2(S,H^1(\Omega)).
\end{align}
It was observed previously that for a.e. $t\in S$, $\Re_U^N=U^N+e_{U,N}^1$ is the orthogonal projection of $U$ onto $\S(\T_N)$ with respect to $\langle\cdot,\cdot\rangle_{H^1(\Omega)}$. Therefore, we also have
\begin{align}
\nonumber
\langle U-\Re_U^N,s\rangle_{L^2(S,H^1(\Omega))}=0
\end{align}
for every $s\in\S(\T_N)$. 
Denote by $P_N$ the operator of orthogonal projection onto $\S(\T_N)$, so that $\Re_U^N=P_NU$.
Further, let $P$ denote the operator of orthogonal projection onto
\begin{align}
\nonumber
\overline{\bigcup_{i=1}^\infty\S(\T_i)}
\end{align}
By Lemma 6.1 in \cite{BabuskaVogelius}, we have
\begin{align}
\label{conv1}
\lim_{N\rightarrow\infty}P_NU=PU
\end{align}
in $L^2(S,H^1(\Omega))$.
We also have $e_{U,N}^2=U-P_NU$, so by (\ref{conv1})
\begin{align}
\lim_{N\rightarrow\infty}\|e^2_{U,N}\|_{L^2(S,H^1(\Omega))}=\|(I-P)U\|_{L^2(S,H^1(\Omega))}.
\end{align}
We want to show that $(I-P)U=0$.

%respectively. Since $U^N=P_NU$, we have
%\begin{align}
%\label{Teo1}
%e_U^N=U-P_NU=U-PU+\xi_N
%\end{align}
%where $\xi_N\rightarrow 0$ in $H^1(\Omega)$ (Lemma 6.1 in Babuska).

%Denote now $R_i$ and $R$ the orthogonal projections onto the spaces
%\begin{align}
%\nonumber
%\{v\in H^1(\Omega): v(x)=0\quad{\rm for}\quad x\in E(\Delta_i)\}
%\end{align}
%and
%\begin{align}
%\nonumber
%\left\{v\in H^1(\Omega): v(x)=0\quad{\rm for}\quad x\in \bigcup E(\Delta_i)\right\}
%%\end{align}
%\framebox{Why are these spaces well-defined? Trace theorem?}

%By (\ref{Teo1}) and Lemma 6.2 in Babuska, we have
%\begin{align}
%\label{Teo2}
%R_Ne_U^N\rightarrow R(I-P)U
%\end{align}
%in $H^1(\Omega)$. \framebox{Perhaps $L^2(S,H^1(\Omega))$ which is a Hilbert space}

For each $m\in\N$, denote by $\widetilde{Q}_m$ one of the squares divided in the transition from $\T_m$ to $\T_{m+1}$ (there is at least one). Since 
\begin{align}
\nonumber
\bigcup_{i=1}^\infty\T_i 
\end{align}
only contains a finite number of different squares with larger areas than a given positive quantity, and, since we exclude repetitions in $\{\widetilde{Q}_m\}$, it follows that
\begin{align}
\nonumber
|\widetilde{Q}_m|\rightarrow 0.
\end{align}
We have
\begin{align}
\nonumber
\nu_{\T_N}(\widetilde{Q}_N)^2=\int_0^T\int_{\widetilde{Q}_N}\left[(\nabla e^2_{U,N})^2+(e^2_{U,N})^2\right]dxdt.
\end{align}
By the fact that $e^2_{U,N}\rightarrow(I-P)U$ in $L^2(S,H^1(\Omega))$ and $|\widetilde{Q}_N|\rightarrow 0$, we have
$\nu_{\T_N}(\widetilde{Q}_N)\rightarrow 0$ by Lebesgue's dominated convergence.

Further, since $\widetilde{Q}_N$ was subdivied, we must have
\begin{align}
\nonumber
\nu_{\T_N}(\widetilde{Q}_N)\ge\beta\max_{Q\in\T_N}\nu_{\T_N}(Q) ,
\end{align}
whence 
\begin{align}
\label{Teo3}
\lim_{N\rightarrow\infty}\max_{Q\in\T_N}\nu_{\T_N}(Q)= 0.
\end{align}
Assume for a contradiction that there exists $x_0\in\Omega$ and an interval $(t_0,t_1)\subset (0,T)$ such that $(I-P)U(t,x_0)\neq0$ for all $t\in (t_0,t_1)$ . For each $N$ denote by $Q_{0,N}\in\T_N$ the square that contains $x_0$. Then there exists a $M\in\N$ such that for all $N\ge M$
\begin{align}
\nonumber
\int_{t_0}^{t_1}[e^2_{N,U}(t,x_0)]^2dt\ge\epsilon>0
\end{align}
At the same time, using the inequality
\begin{align}
\nonumber
|e^2_{U,N}(t,x_0)|^2\le C\|\nabla e^2_{U,N}\|_{L^2(Q_{N,0})}\|e^2_{U,N}\|_{L^2(Q_{N,0})}\quad(0<t<T)
\end{align}
(see \cite{LadyzenskajaEtAl}), we get
\begin{eqnarray}
\nonumber
0<\epsilon&\le&\left(\int_0^T\|e^2_{U,N}(t)\|^2_{L^2(Q_{0,N})}dt\right)^{1/2}\left(\int_0^T\|\nabla e^2_{U,N}(t)\|^2_{L^2(Q_{0,N})}dt\right)^{1/2}\\
\nonumber
&\le& 2\int_0^T\left[\|e^2_{U,N}(t)\|^2_{L^2(Q_{0,N})}+\|\nabla e^2_{U,N}(t)\|^2_{L^2(Q_{0,N})}\right]dt\\
\label{Teo4}
&=&2\nu_{\T_N}(Q_{0,N})\le C\max_{Q\in\T_N}\nu_{\T_N}(Q).
\end{eqnarray}
But (\ref{Teo4}) contradicts (\ref{Teo3}), hence $(I-P)U=0$ and
\begin{align}
\nonumber
\lim_{N\rightarrow\infty}\|e^2_{U,N}\|_{L^2(S,H^1(\Omega))}=0.
\end{align}
\end{proof}

Let $\{\mathcal{T}_i\}$ be as above, $\mathcal{T}_Y$ an arbitrary partition of $Y$ and $\widetilde{U}_i,\widetilde{v},\widetilde{w}$ be associated Galerkin projections. Then we have the following convergence result.
\begin{teo}
\label{mainTeo} 
With the notation above, we have
\begin{align}
\|e_{U,\mathcal{T}_N}\|^2_{L^2(S,H^1(\Omega))}+\|e_v\|^2_\cX+\|e_w\|^2_\cX=\O(h_Y^2)+\e_N
\end{align}
where $\lim_{N\rightarrow\infty}\e_N=0$.
\end{teo}
\begin{proof}
It follows from (\ref{errLimit0}) and (\ref{errLimit}) that
\begin{align}
\nonumber
\e_N=\|e_{U,\mathcal{T}_N}\|_{L^2(S,H^1(\Omega))}\rightarrow0
\end{align}  
as $N\rightarrow\infty$. Further, by (\ref{mainErrorEstimate}) we have
\begin{align}
\nonumber
\|e_U\|^2_{L^2(S,H^1(\Omega))}+\|e_v\|^2_\cX+\|e_w\|^2_\cX\leq \e_N+\O(h^2_Y),
\end{align}
this concludes the proof.
\end{proof}

%%%%%%%%%%%%%%%%%%%%%%%%%%%%%%%%%%%%%%%%%%%%%%%%%%%%%%%%%%%%%%%%%%%%%%%%%%
\section*{Appendix A: Higher regularity of the weak solution to (\ref{system})-(\ref{systemInitial})}
In this appendix, we shall sketch a proof of Theorem \ref{highRegularity}. We follow closely the idea used in \cite[Theorem 7.1.5, p. 350]{Evans}; an alternative route could use the Nirenberg method based on difference quotients or some other classical technique pointing out the regularity lift. 
\begin{proof}[Proof of Theorem \ref{highRegularity}]
Let $\T_\Omega,\T_Y$ be any partitions and $(\tU,\tv,\tw)$ solves (\ref{GalU})-(\ref{GalW}). As in \cite[Theorem 7.1.5, p. 350]{Evans}, we are essentially done if we prove that
\begin{align}
\label{L^2TimeDerivative}
(\partial_tU,\partial_tv,\partial_tw)\in L^2(S,H^1(\Omega))\times [L^2(S,L^2(\Omega,H^1(Y)))]^2.
\end{align}
To simplify notation, we define the function
\begin{align}
\nonumber
J(g,t)=\int_{\Omega\times\Gamma^R}(\tv-\tU)gd\sigma_ydx
\end{align}
for any function $g=g(t,x,y)$ which is well-defined and Lebesgue integrable on $\Omega\times\Gamma^R$.

Testing with $(\partial_t\tU,\partial_t\tv,\partial_t\tw)\in\S(\T_\Omega)\times\S(\T_Y)^2$ in  (\ref{GalU})-(\ref{GalW}) and adding the equations, we obtain
\begin{align}
\nonumber
\|\partial_t\tU\|^2_{L^2(\Omega)}+\|\partial_t\tv\|^2_{L^2(\Omega\times Y)}+\|\partial_t\tw\|^2_{L^2(\Omega\times Y)}\hspace{3cm}\\
\nonumber
+\frac{d}{2dt}\left(\|\nabla\tU\|^2_{L^2(\Omega)}+\|\nabla_y\tv\|^2_{L^2(\Omega\times Y)}+\|\nabla_y\tw\|^2_{L^2(\Omega\times Y)}\right)+\a J(\partial_t\tv,t)\\
\label{regLift3}
\le|\a\gamma||J(\partial_t\tU,t)|+C\int_{\Omega\times Y}|\eta(\tv,\tw)|[|\partial_t\tv|+|\partial_t\tw|]
\end{align}
The term $J(\partial_t\tv,t)$ requires some analysis. Note that
\begin{eqnarray}
\nonumber
J(\partial_t\tv,t)&=&\int_{\Omega\times\Gamma^R}(\tv-\tU)\partial_t\tv d\sigma_ydx=\frac{d}{2dt}\int_{\Omega\times\Gamma^R}\tv^2d\sigma_ydx\\
\nonumber
&-&\int_{\Omega\times\Gamma^R}\tU\partial_t\tv d\sigma_ydx
\end{eqnarray}
Further,
\begin{align}
\nonumber
\int_{\Omega\times\Gamma^R}\tU\partial_t\tv d\sigma_ydx=\frac{d}{dt}\int_{\Omega\times\Gamma^R}\tU\tv d\sigma_ydx-\int_{\Omega\times\Gamma^R}\partial_t\tU\tv d\sigma_ydx
\end{align}
Consequently, denoting 
\begin{align}
\nonumber
\Theta(t)=\int_{\Omega\times\Gamma^R}(\tv^2-2\tU\tv)d\sigma_ydx, 
\end{align}
the estimate (\ref{regLift3}) can be written
\begin{align}
\nonumber
\|\partial_t\tU\|^2_{L^2(\Omega)}+\|\partial_t\tv\|^2_{L^2(\Omega\times Y)}+\|\partial_t\tw\|^2_{L^2(\Omega\times Y)}\hspace{3.0cm}\\
\nonumber
+\frac{d}{2dt}\left(\|\nabla\tU\|^2_{L^2(\Omega)}+\|\nabla_y\tv\|^2_{L^2(\Omega\times Y)}+\|\nabla_y\tw\|^2_{L^2(\Omega\times Y)}+\Theta\right)\hspace{1.0cm}\\
\le 
\label{regLift4}
C\left[|J(\partial_t\tU,t)|+\int_{\Omega\times\Gamma^R}|\tv\partial_t\tU|d\sigma_ydx+\int_{\Omega\times Y}|\eta(\tv,\tw)|[|\partial_t\tv|+|\partial_t\tw|]\right].
\end{align}
We estimate the right-hand side of (\ref{regLift4}).
By (\ref{epsilonAMGM}) with $\e>0$ and (\ref{InterpolationTrace}) with $\rho=1/2$, we have
\begin{align}
\nonumber
J(\partial_t\tU,t)\le\e|\Gamma^R|\|\partial_t\tU\|^2_{L^2(\Omega)}+c_\e\int_{\Omega\times\Gamma^R}(\tv-\tU)^2d\sigma_ydx\\
\label{regLift5}
\le\e|\Gamma^R|\|\partial_t\tU\|^2_{L^2(\Omega)}+C\left[\|\nabla_y\tv\|^2_{L^2(\Omega\times Y)}+\|\tU\|^2_{L^2(\Omega)} +\|\tv\|^2_{L^2(\Omega\times Y)} \right].
\end{align}
In the same way,
\begin{align}
\nonumber
\int_{\Omega\times\Gamma^R}|\tv\partial_t\tU|d\sigma_ydx\le\e|\Gamma^R|\|\partial_t\tU\|^2_{L^2(\Omega)}+\\
\label{regLift6}
+C\left[\|\nabla_y\tv\|^2_{L^2(\Omega\times Y)}+\|\tv\|^2_{L^2(\Omega\times Y)} \right]
\end{align}
Finally, using the Lipschitz continuity of $\eta$, the last term of (\ref{regLift4}) can be estimated 
\begin{align}
\nonumber
\int_{\Omega\times Y}|\eta(\tv,\tw)|[|\partial_t\tv|+|\partial_t\tw|]\le \frac{1}{2}\left(\|\partial_t\tv\|^2_{L^2(\Omega\times Y)}+\|\partial_t\tw\|^2_{L^2(\Omega\times Y)}\right)\\
\label{regLift7}
+C\left(\|\tv\|^2_{L^2(\Omega\times Y)}+\|\tw\|^2_{L^2(\Omega\times Y)}\right).
\end{align}
In summary, by choosing $\e=(4|\Gamma^R|)^{-1}$ in (\ref{regLift5}) and (\ref{regLift6}) and by using estimates (\ref{regLift4})-(\ref{regLift7}), we have
\begin{align}
\nonumber
\|\partial_t\tU\|^2_{L^2(\Omega)}+\|\partial_t\tv\|^2_{L^2(\Omega\times Y)}+\|\partial_t\tw\|^2_{L^2(\Omega\times Y)}+\\
\nonumber
+\frac{d}{dt}\left(\|\nabla\tU\|^2_{L^2(\Omega)}+\|\nabla_y\tv\|^2_{L^2(\Omega\times Y)}+\|\nabla_y\tw\|^2_{L^2(\Omega\times Y)}+\Theta\right)\\
\nonumber
\le C\left[\|\nabla\tU\|^2_{L^2(\Omega)}+\|\nabla_y\tv\|^2_{L^2(\Omega\times Y)}+\|\nabla_y\tw\|^2_{L^2(\Omega\times Y)}+\right.\\
\nonumber
\left.+\|\tU\|^2_{L^2(\Omega)}+\|\tv\|^2_{L^2(\Omega\times Y)}+\|\tw\|^2_{L^2(\Omega\times Y)}\right].
\end{align}
Set
\begin{align}
\nonumber
\Upsilon(t):=\|\partial_t\tU(t)\|^2_{L^2(\Omega)}+\|\partial_t\tv(t)\|^2_{L^2(\Omega\times Y)}+\|\partial_t\tw(t)\|^2_{L^2(\Omega\times Y)},
\end{align}
\begin{align}
\nonumber
\Xi(t):=\|\nabla\tU(t)\|^2_{L^2(\Omega)}+\|\nabla_y\tv(t)\|^2_{L^2(\Omega\times Y)}+\|\nabla_y\tw(t)\|^2_{L^2(\Omega\times Y)}+\Theta(t)
\end{align}
and
\begin{align}
\nonumber
\mu(t):=\|\tU(t)\|^2_{L^2(\Omega)}+\|\tv(t)\|^2_{L^2(\Omega\times Y)}+\|\tw(t)\|^2_{L^2(\Omega\times Y)},
\end{align}
then estimating $\Theta$ using (\ref{InterpolationTrace}), we obtain that for all $t\in S$
\begin{align}
\nonumber
\Upsilon(t)+\Xi'(t)\le C\Xi(t)+C\mu(t). 
\end{align}
By Gr\"{o}nwall's inequality, we see that
\begin{align}
\nonumber
\Xi(t)\le C\left[\Xi(0)+\int_0^T\mu(s)ds\right].
\end{align}
Clearly $\Xi(0)=C_I$ is a finite constant depending only on the Galerkin approximation of the initial values. Further, we get
\begin{align}
\nonumber
\int_0^T\left[\Upsilon(t)+\Xi(t)\right]dt\le C_I+C\int_0^T\mu(t)dt.
\end{align}
In other words, we have
\begin{align}
\nonumber
\|\partial_t\tU\|^2_{L^2(S\times\Omega))}+\|\partial_t\tv\|^2_{L^2(S\times\Omega\times Y))}+\|\partial_t\tw\|^2_{L^2(S\times\Omega\times Y)}\\
\nonumber
+\|\nabla\tU\|^2_{L^2(S,L^2(\Omega))}+\|\nabla_y\tv\|^2_{L^2(S,L^2(\Omega\times Y))}+\|\nabla_y\tw\|^2_{L^2(S,L^2(\Omega\times Y))}\\
\label{regLift8}
\le C_I+\|\tU\|^2_{L^2(S\times\Omega))}+\|\tv\|^2_{L^2(S\times\Omega\times Y)}+\|\tw\|^2_{L^2(S\times\Omega\times Y)}
\end{align}
Using the strong convergence in $L^2$ of $(\tU,\tv,\tw)$, we get
\begin{align}
\nonumber
(\partial_t\tU,\partial_t\tv,\partial_t\tw)\rightharpoonup(\partial_t U,\partial_t v,\partial_t w)\in  L^2(S,H^1(\Omega))\times [L^2(S,L^2(\Omega,H^1(Y)))]^2
\end{align}
as $\max(h_\Omega,h_Y)\rightarrow0$.
\end{proof}

\section*{Acknowledgements}
We are grateful to Michael Eden (Bremen) for useful discussions.

\end{document}